\documentclass[10pt]{amsart}

\usepackage{amssymb,latexsym,amsmath,epsfig,amsthm} 
\usepackage[utf8]{inputenc} 
\usepackage{amsbsy}
\usepackage{amscd}
\usepackage{verbatim}
\usepackage[english]{babel}
\usepackage{dsfont}
\usepackage{anysize}
\usepackage{hyperref}
\usepackage{leftidx}
\usepackage{footnote}
\usepackage{tablefootnote}

\newcommand{\e}{\varepsilon}

\newcommand{\tl}{\text{li}}

\newcommand{\q}{\quad}

\newtheorem{thm}{Theorem}[section]
\newtheorem{lem}[thm]{Lemma}

\newtheorem{prop}[thm]{Proposition}
\newtheorem{kor}[thm]{Corollary}

\theoremstyle{definition}

\theoremstyle{remark}
\newtheorem*{rema}{Remark}

\title{Effective estimates for some functions defined over primes}

\author{Christian Axler}
\address{Institute of Mathematics\\ Heinrich-Heine-University Düsseldorf\\
40225 Düsseldorf, Germany}
\email{christian.axler@hhu.de}

\date{\today}
\subjclass[2010]{11N56 (Primary), 11N05, 11A41 (Secondary)}
\keywords{Chebyshev’s $\vartheta$-function, prime counting function, primes in short intervals}

\begin{document}

\begin{abstract}
In this paper we give effective estimates for some classical arithmetic functions defined over prime numbers. First we find the smallest real number $x_0$ so that some inequality involving Chebyshev's $\vartheta$-function holds for every $x \geq x_0$. Then we give some new results concerning the existence of prime numbers in short intervals. Also we derive new upper and lower bounds for some functions defined over prime numbers, for instance the prime counting function $\pi(x)$, which improve current best estimates of similar shape.
\end{abstract}

\maketitle

\section{Introduction}

First, we consider Chebyshev's $\vartheta$-function
\begin{displaymath}
\vartheta(x) = \sum_{p \leq x} \log p,
\end{displaymath}
where $p$ runs over all primes not exceeding $x$. Since there are infinitely many primes, we have $\vartheta(x) \to \infty$ as $x \to \infty$.
Hadamard \cite{hadamard1896} and de la Vall\'{e}e-Poussin \cite{vallee1896} independently proved a result concerning the asymptotic behavior for $\vartheta(x)$, namely
\begin{equation}
\vartheta(x) \sim x \q\q (x \to \infty), \tag{1.1} \label{1.1}
\end{equation}
which is known as the \textit{Prime Number Theorem}. In a later paper \cite{vallee1899}, where the existence of a zero-free region for the Riemann zeta function to the left of the line $\text{Re}(s) = 1$ was proved, de la Vall\'{e}e-Poussin also estimated the error term in the Prime Number Theorem by showing that
\begin{equation}
\vartheta(x) = x + O(x e^{-c_0\sqrt{\log x}}) \q\q (x \to \infty), \tag{1.2} \label{1.2}
\end{equation} 
where $c_0$ is a positive absolute constant. The currently best explicit version of this result is due to Johnston and Yang \cite[Corollary 1.2]{johnstonyang}. They found that
\begin{equation}
|\vartheta(x) - x| \leq 9.14 x (\log x)^{1.515} \exp(- 0.8274\sqrt{\log x} ) \tag{1.3} \label{1.3}
\end{equation}
for every $x \geq 2$. The work of Korobov \cite{korobov1958} and Vinogradov \cite{vinogradov1958} implies the currently asymptotically strongest error term in \eqref{1.1}, namely
\begin{equation}
\vartheta(x) = x + O \left( x \exp \left( - c_1 \log^{3/5}x (\log \log x)^{-1/5} \right) \right) \q\q (x \to \infty), \tag{1.4} \label{1.4}
\end{equation}
where $c_1$ is a positive absolute constant. An explicit version of \eqref{1.4} was recently given by Johnston and Yang \cite[Theorem 1.4]{johnstonyang}. Now, \eqref{1.2}--\eqref{1.4} each imply that for every positive integer $k$ and every positive real number $\eta_k$ there is real number $x_1 = x_1(k, \eta_k) > 1$ so that
\begin{equation}
|\vartheta(x) - x| < \frac{\eta_kx}{\log^k x} \tag{1.5} \label{1.5}
\end{equation}
for every $x \geq x_1$. In the case where $k=3$ and $\eta_3 = 0.024334$, Broadbent et al. \cite[Table 15]{kadiri} found that
\begin{equation}
|\vartheta(x) - x| < \frac{0.024334x}{\log^3x} \q\q (x \geq e^{29}). \tag{1.6} \label{1.6}
\end{equation}
In our first result, we compute the smallest positive integer $N$ so that \eqref{1.6} holds for every $x \geq N$.

\begin{prop} \label{prop101}
The inequality \eqref{1.6} holds for every $x \geq 1,757,126,630,797 = p_{64,707,865,143}$.
\end{prop}

Estimates for $\vartheta(x)$ of the form \eqref{1.5} can be used to specify short intervals containing at least one prime number. As an application of the Proposition \ref{prop101} and a recent result of Broadbent et al. \cite{kadiri}, we find the following result concerning the existence of prime numbers in short intervals.

\begin{thm} \label{thm102}
For every $x \geq x_0$ there is a prime number $p$ such that
\begin{displaymath}
x < p \leq x \left( 1 + \frac{k}{\log^nx} \right),
\end{displaymath}
where
\begin{center}
\begin{tabular}{|l||c|c|c|c|c|c|c|c|c|c|}
\hline
$n$\rule{0mm}{4mm}   & $ 3                              $ & $   4              $ & $      5            $ \\ \hline
$k$\rule{0mm}{4mm}   & $ 0.048668 + 8.22 \times 10^{-11}$ & $ 114.368 + 2.145 \times 10^{-10}$ & $268,820 + 5.0363 \times 10^{-7}$ \\ \hline 
$x_0$\rule{0mm}{4mm} & $ 17,051,708                     $ & $   2              $ & $      2            $ \\ \hline
\end{tabular} \ .
\end{center}
\vspace*{2mm}
\end{thm}


Let $\pi(x)$ denote the number of primes not exceeding $x$. Chebyshev's $\vartheta$-function and the prime counting function $\pi(x)$ are connected by the 
identity
\begin{equation}
\pi(x) = \frac{\vartheta(x)}{\log x} + \int_{2}^{x}{\frac{\vartheta(t)}{t \log^{2} t}\ \text{d}t}, \tag{1.7} \label{1.7}
\end{equation}
which holds for every $x \geq 2$ (see \cite[Theorem 4.3]{ap}). If we combine \eqref{1.3} and \eqref{1.7}, we see that
\begin{equation}
\pi(x) = \text{li}(x) + O(x e^{-c_2\sqrt{\log x}}) \q\q (x \to \infty), \tag{1.8} \label{1.8}
\end{equation} 
where $c_2$ is a positive absolute constant. Here, the \textit{integral logarithm} $\text{li}(x)$ is defined for every $x \geq 0$ as
\begin{displaymath}
\text{li}(x) = \int_0^x \frac{\text{d}t}{\log t} = \lim_{\varepsilon \to 0+} \left \{ \int_{0}^{1-\varepsilon}{\frac{\text{d}t}{\log t}} + 
\int_{1+\varepsilon}^{x}{\frac{\text{d}t}{\log t}} \right \}
\end{displaymath}
and plays an important role in this paper. The current best explicit version of \eqref{1.8} is due to Johnston and Yang \cite[Corollary 1.3]{johnstonyang}.
Again, the work of Korobov \cite{korobov1958} and Vinogradov \cite{vinogradov1958} implies the current asymptotically strongest error term for the difference $\pi(x) - \text{li}(x)$, namely
\begin{equation}
\pi(x) = \text{li}(x) + O \left( x \exp \left( - c_3 (\log x)^{3/5} (\log \log x)^{-1/5} \right) \right) \q\q (x \to \infty), \tag{1.9} \label{1.9}
\end{equation}
where $c_3$ is a positive absolute constant. Ford \cite[p.\:2]{ford} has found that the constant $c_3$ in \eqref{1.9} can be chosen to be equal to $0.2098$. Johnston and Yang \cite[Theorem 1.4]{johnstonyang} used explicit zero-free regions and zero-density estimates for the Riemann zeta-function
to show that the inequality
\begin{equation}
|\pi(x) - \text{li}(x)| \leq 0.028x(\log x)^{0.801} \exp \left( - 0.1853 (\log x)^{3/5} (\log \log x)^{-1/5} \right) \tag{1.10} \label{1.10}
\end{equation}
holds for every $x \geq 71$.
Panaitopol \cite[p.\:55]{pana} gave another completely different asymptotic formula for the prime counting function by showing that for every positive integer $m$, one has
\begin{equation}
\pi(x) = \frac{x}{ \log x - 1 - \frac{k_1}{\log x} - \frac{k_2}{\log^2x} - \ldots - \frac{k_{m}}{\log^{m} x}} + O \left( \frac{x}{\log^{m+2}x} \right) \q\q (x \to \infty), \tag{1.11} \label{1.11}
\end{equation}
where the positive integers $k_1, \ldots, k_m$ are defined by the recurrence formula
\begin{displaymath}
k_m + 1!k_{m-1} + 2!k_{m-2} + \ldots + (m-1)!k_1 = m \cdot m!.
\end{displaymath}
For instance, we have $k_1 = 1$, $k_2 = 3$, $k_3 = 13$, $k_4 = 71$, $k_5 = 461$, and $k_6 = 3441$. The computation of the prime counting function $\pi(x)$ for large values of $x$ is a difficult problem (the latest record is due to Baugh and Walisch and was $\pi(10^{28}) = 
157,589,269,275,973,410,412,739,598$). Also the asymptotic formula \eqref{1.8} (or \eqref{1.11}) is not very meaningful with regard to the computation of $\pi(x)$ for some fixed $x$. Hence we are interested in finding new effective estimates for the prime counting function $\pi(x)$ which correspond to the first terms of \eqref{1.11}. For instance, those estimates for the prime counting function are used to get effective estimates for $1/\pi(x)$ (see \cite{berkanedusart}) or the $n$th prime number (see \cite{axlernthp}). In this paper, we use Proposition \ref{prop101} to establish the following upper bound for $\pi(x)$ which corresponds to the first terms of the asymptotic formula \eqref{1.11}.

\begin{thm} \label{thm103}
For every $x \geq 48$, we have
\begin{equation}
\pi(x) < \frac{x}{\log x - 1 - \frac{1}{\log x} - \frac{3.024334}{\log^2x} - \frac{12.975666}{\log^3x} - \frac{71.048668}{\log^4x} - 
\frac{461.364417856444}{\log^5x} - \frac{4331.1}{\log^6x}}. \tag{1.12} \label{1.12}
\end{equation}
\end{thm}

For all sufficiently large values of $x$, Theorem \ref{thm103} is a consequence of
\eqref{1.10}. On the other hand, we get the following lower bound for the $\pi(x)$ which corresponds to the first terms of \eqref{1.11}.

\begin{thm} \label{thm104}
For every $x \geq 1,751,189,194,177 = p_{64,497,259,289}$, we have
\begin{equation}
\pi(x) > \frac{x}{\log x - 1 - \frac{1}{\log x} - \frac{2.975666}{\log^2x} - \frac{13.024334}{\log^3x} - \frac{70.951332}{\log^4x} - 
\frac{460.634397856444}{\log^5x} - \frac{3444.031844143556}{\log^6x}}. \tag{1.13} \label{1.13}
\end{equation}
\end{thm}

Again, for all sufficiently large values of $x$, Theorem \ref{thm104} follows directly from
\eqref{1.10}. Our next goal is to establish new explicit estimates for the functions
\begin{displaymath}
\sum_{p \leq x} \frac{1}{p} \q \text{and} \q \sum_{p \leq x} \frac{\log p}{p},
\end{displaymath}
where $p$ runs over primes not exceeding $x$, respectively. Euler \cite{euler1737} proved that the sum of the reciprocals of all prime numbers
diverges.
Mertens \cite[p.\:52]{mertens1874} found that $\log \log x$ is the right order of magnitude for this sum
by showing 
\begin{equation}
\sum_{p \leq x} \frac{1}{p} = \log \log x + B + O \left( \frac{1}{\log x} \right). \tag{1.14} \label{1.14}
\end{equation}
Here $B$ denotes \textit{the Mertens' constant}
and is defined by
\begin{equation}
B = \gamma + \sum_{p} \left( \log \left( 1 - \frac{1}{p} \right) + \frac{1}{p} \right) = 0.26149 \ldots , \tag{1.15} \label{1.15}
\end{equation}
where $\gamma = 0.577215\ldots$ denotes the Euler-Mascheroni constant. In Section 6, we apply Proposition \ref{prop101}
to some identity obtained by Rosser and Schoenfeld \cite{rosserschoenfeld1962} and derive the following result which improves all other results of this form.

\begin{thm} \label{thm105}
For every $x \geq 1,757,126,630,797$, we have
\begin{equation}
\left| \sum_{p \leq x} \frac{1}{p} - \log \log x - B \right| \leq \frac{0.024334}{3\log^3 x} \left( 1 + \frac{15}{4\log x} \right). \tag{1.16} \label{1.16}
\end{equation}
\end{thm}

In 1874, Mertens \cite{mertens1874} showed that
\begin{equation}
\sum_{p \leq x} \frac{\log p}{p} = \log x + O(1). \tag{1.17} \label{1.17}
\end{equation}
Landau \cite[§55]{landau} improved \eqref{1.17} by finding
\begin{displaymath}
\sum_{p \leq x} \frac{\log p}{p} = \log x + E + O (\exp(-\sqrt[14]{\log x})),
\end{displaymath}
where $E$ is a constant 
defined by
\begin{equation}
E = - \gamma - \sum_{p} \frac{\log p}{p(p-1)} = -1.3325 \ldots. \tag{1.18} \label{1.18}
\end{equation}
Similar to Theorem \ref{thm105}, we establish the following explicit estimates for $\sum_{p \leq x} \log(p)/p$ which improve \cite[Proposition 8]{axlerPrimeF}.

\begin{thm} \label{thm106}
For every $x \geq 1,757,126,630,797$, we have
\begin{displaymath}
\left| \sum_{p \leq x} \frac{\log p}{p} - \log x - E \right| \leq \frac{0.024334}{2\log^2x} \left( 1 + \frac{2}{\log x} \right).
\end{displaymath}
\end{thm}

\section{Proof of Proposition \ref{prop101}}

In the following proof of Proposition \ref{prop101}, we first utilize an identity investigated by Rosser and Schoenfeld \cite{rosserschoenfeld1962} to express Chebyshev's $\vartheta$-function in terms of the difference $\pi(x) - \text{li}(x)$. Then we apply Walisch's \textit{primecount} C++ code \cite{walisch} to find a lower bound for $\pi(x) - \text{li}(x)$ in a certain restricted interval.

\begin{proof}[Proof of Proposition \ref{prop101}]
By \eqref{1.6} and \cite[Corollary 11.1]{kadiri},
it suffices to check that the inequality
\begin{equation}
\vartheta(x) > x - \frac{0.024334x}{\log^3 x} \tag{2.1} \label{2.1}
\end{equation}
holds for every $x$ satisfying $1,757,126,630,797 \leq x \leq e^{29}$. Using \cite[(2.26)]{rosserschoenfeld1962} with $f(x) = \log x$, we get
\begin{equation}
\vartheta(x) = x-2 + \tl(2) \log 2 + (\pi(x) - \tl(x)) \log x - \int_2^x \frac{\pi(t) - \tl(t)}{t} \, \text{d}t \tag{2.2} \label{2.2}
\end{equation}
for every $x \geq 2$. Now we can use \cite[Corollary 1]{platttrudskewes} to see that
\begin{equation}
-2 + \tl(2) \log 2 - \int_2^x \frac{\pi(t) - \tl(t)}{t} \, \text{d}t \geq -2 + \tl(2) \log 2 - \int_2^9 \frac{\pi(t) - \tl(t)}{t} \, \text{d}t \geq 0.129 \tag{2.3} \label{2.3}
\end{equation}
for every $x$ with $9 \leq x \leq e^{29}$. Applying \eqref{2.4} to \eqref{2.2}, we get
\begin{equation}
\vartheta(x) > x + (\pi(x) - \tl(x)) \log x \tag{2.4} \label{2.4}
\end{equation}
for every $x$ so that $9 \leq x \leq e^{29}$. Now we use Walisch's \textit{primecount} C++ code \cite{walisch} to get
\begin{equation}
\pi(x) - \text{li}(x) \geq - \frac{0.024334x}{\log^4x} \tag{2.5} \label{2.5}
\end{equation}
for every $x$ with $1,760,505,892,241 \leq x \leq 2,342,911,050,819$ and every $x$ with $2,346,094,807,193 \leq x \leq 4 \times 10^{12}$. If we combine \eqref{2.5} with \eqref{2.4}, we get \eqref{2.1} for every $x$ satisfying $1,760,505,892,241 \leq x \leq 2,342,911,050,819$ and every $x$ with $2,346,094,807,193 \leq x \leq e^{29} \leq 4 \times 10^{12}$. In order to verify the required inequality \eqref{2.1} in the case where $x$ satisfies $1,757,126,630,797 \leq x < 1,760,505,892,241$, we can check with a computer that $\vartheta(p_n) > g(p_{n+1})$ for every integer $n$ such that $\pi(1,757,126,630,797) \leq n \leq \pi(1,760,505,892,241)$. Finally, a direct computer check shows that the inequality \eqref{2.1} also holds for every $x$ such that $2,342,911,050,819 \leq x \leq 2,346,094,807,193$.
\end{proof}

\begin{rema}
To find other explicit estimates for $\vartheta(x)$ in the restricted interval $[2, 10^{20}]$, one can also apply the method used by Dusart in \cite{dusart20182}. Let
\begin{displaymath}
\pi_0(x) = \lim_{\e \to 0} \frac{\pi(x-\e) + \pi(x + \e)}{2} =
\begin{cases}
\pi(x) - 1/2, &\text {if $x$ is prime,} \\
\pi(x), &\text {otherwise.} \nonumber
\end{cases}
\end{displaymath}
Riemann \cite{riemann} published the formula
\begin{equation}
\pi_0(x) = \sum_{n=1}^\infty \frac{\mu(n)}{n} \, f(x^{1/n}), \tag{2.6} \label{2.6}
\end{equation}
where $\mu(n)$ is the Möbius function, and $f(x)$ is the Riemann prime counting function
\begin{displaymath}
f(x) = \tl{(x)} - \sum_{\rho} \tl{(x^{\rho})} + \int_x^{\infty} \frac{\text{d}t}{t(t^2-1)\log t} - \log 2.
\end{displaymath}
Here the sum means $\lim_{T \to \infty} \sum_{|\rho| \leq T} \tl{(x^{\rho})}$, and the $\rho$'s are the nontrivial zeros of the Riemann zeta function.
A first proof of \eqref{2.6} was given by von Mangoldt \cite{mangoldt} in 1895. Now let
\begin{equation}
R(x) = \sum_{n=1}^\infty \frac{\mu(n)}{n} \, \tl{(x^{1/n})} = 1 + \sum_{k=1}^\infty \frac{\log^k x}{k!k \zeta(k+1)}. \tag{2.7} \label{2.7}
\end{equation}
The latter series for it is known as Gram series. Since $\log x < x$ for every real $x > 0$, this series converges for all positive $x$ by comparison with the series for $e^x$. In \cite{riesel}, Riesel and Göhl showed that the function
\begin{displaymath}
g(x) = R(x) - \frac{1}{\log x} + \frac{1}{\pi} \, \arctan \frac{\pi}{\log x}
\end{displaymath}
is a quite good approximation to $\pi_0(x)$. The difference between $g(x)$ and $\pi_0(x)$ heuristically oscillates with an amplitude of about $\sqrt{x}/\log x$. So we define
\begin{equation}
\Delta(x) = \left( \pi_0(x) - R(x) + \frac{1}{\log x} - \frac{1}{\pi} \, \arctan \frac{\pi}{\log x} \right) \frac{\log x}{\sqrt{x}}, \tag{2.8} \label{2.8}
\end{equation}
the function which represents the fluctuations of the distribution of primes. 
We can use \eqref{2.7} and \eqref{2.8} to get
\begin{equation}
\pi(x) - \tl{(x)} \leq \frac{1}{2} + f_2(x) + \frac{\sqrt{x}}{\log x} \times \Delta(x) - \frac{1}{\log x} + \frac{1}{\pi} \, \arctan \frac{\pi}{\log x}, \tag{2.9} \label{2.9}
\end{equation}
where
\begin{displaymath}
f_k(x) = \sum_{n=k}^\infty \frac{\mu(n)}{n} \, \text{li}(x^{1/n}).
\end{displaymath}
Since $\mu(4) = 0$ and $f_5(x)$ is strictly decreasing on $(1, \infty)$, the inequality \eqref{2.9} implies that
\begin{equation}
\pi(x) - \tl{(x)} \leq - \frac{\tl{(\sqrt{x})}}{2} - \frac{\tl{(x^{1/3})}}{3} + \frac{\sqrt{x}}{\log x} \times \Delta(x) \tag{2.10} \label{2.10}
\end{equation}
for every $x \geq 2,000$. Similarly, we see that
\begin{equation}
\pi(x) - \tl{(x)} \geq \sum_{n=2}^5 \frac{\mu(n)}{n} \, \tl{(x^{1/n})} + \frac{\sqrt{x}}{\log x} \times \Delta(x) \tag{2.11} \label{2.11}
\end{equation}
for every $x \geq 10,326$. Applying \eqref{2.10} and \eqref{2.11} to \eqref{2.2}, we get
\begin{displaymath}
\vartheta(x) > x + (\Delta(x)-1) \sqrt{x} - \max_{2000 \leq t \leq x} \Delta(t) \times \text{li}(\sqrt{x}) - \sqrt[3]{x} - \frac{\text{li}(\sqrt[5]{x}) \log x}{5} + c_1
\end{displaymath}
for every $x \geq 10,326$, where $c_1$ is a constant. Analogously, we see that the inequality
\begin{displaymath}
\vartheta(x) < x + (\Delta(x)-1) \sqrt{x} - \min_{10,236 \leq t \leq x} \Delta(t) \times \text{li}(\sqrt{x}) - \sqrt[3]{x} + \frac{\text{li}(\sqrt[5]{x}) \log x}{5} - \sqrt[5]{x} + c_2
\end{displaymath}
holds for every $x \geq 10,326$, where $c_2$ is a constant. Now one can use the extensive table of the minimum and maximum values of $\Delta(x)$ in \cite{kulsha} to obtain explicit estimates for $\vartheta(x)$ in the restricted interval $[2, 10^{20}]$.
\end{rema}

\begin{rema}
Under the assumption that the Riemann hypothesis is true, von Koch \cite{koch1901} deduced the asymptotic formula $\vartheta(x) = x + O(\sqrt{x} \log^2 x)$. An explicit version was given by Schoenfeld \cite[Theorem 10]{schoenfeld1976}. Under the assumption that the Riemann hypothesis is true, Schoenfeld has found that
\begin{equation}
|\vartheta(x) - x| < \frac{\sqrt{x}}{8\pi} \, \log^2 x \tag{2.12} \label{2.12}
\end{equation}
for every $x \geq 599$. Recently, Schoenfeld's result was slightly improved by Dusart \cite[Proposition 2.5]{dusart20182}. In 2016, Büthe \cite[Theorem 2]{buethe2016} investigated a method to show that the inequality \eqref{2.12} holds unconditionally for every $x$ such that $599 \leq x \leq 1.4 \times 10^{25}$. Büthe's result was improved by Platt and Trudgian \cite[Corollary 1]{plattriemann}. They proved that the inequality \eqref{2.12} holds unconditionally for every $x$ satisfying $599 \leq x \leq 2.169 \times 10^{25}$. Recently, Johnston \cite[Corollary 3.3]{johnston} extended the last result by showing that the inequality \eqref{2.12} holds unconditionally for every $x$ with $599 \leq x \leq 1.101 \times 10^{26}$.
\end{rema}

\section{Proof of Theorem \ref{thm102}}

Bertrand's postulate states that for each positive integer $n$ there is a prime number $p$ with $n < p \leq 2n$. It was proved, for instance, by Chebyshev \cite{cheby1850}. In the following, we note some improvements of Bertrand's postulate. The first result is due to
Trudgian \cite[Corollary 2]{trudprime}. He proved that for every $x \geq 2,898,242$ there exists a prime number $p$ with
\begin{equation}
x < p \leq x\left( 1 + \frac{1}{111 \log^2 x} \right). \tag{3.1} \label{3.1}
\end{equation}
Dusart \cite[Corollary 5.5]{dusart20181}
improved Trudgian's result by showing that for every $x \geq 468,991,632$ there exists a prime number $p$ such that
\begin{equation}
x < p \leq x\left( 1 + \frac{1}{5,000 \log^2 x} \right). \tag{3.2} \label{3.2}
\end{equation}
In \cite[Theorem 4]{axlerPrimeF}, it is shown that for every $x \geq 6,034,256$ there exists a prime number $p$ such that
\begin{equation}
x < p \leq x\left( 1 + \frac{0.087}{\log^3 x} \right). \tag{3.3} \label{3.3}
\end{equation}
Further, the present author \cite[Theorem 4]{axlerPrimeF} found that for every $x > 1$ there is a prime number $p$ with
\begin{equation}
x < p \leq x\left( 1 + \frac{198.2}{\log^4 x} \right). \tag{3.4} \label{3.4}
\end{equation}
In Theorem \ref{thm102}, we give improvements of \eqref{3.3} and \eqref{3.4} by decreasing the coefficient of the term $1/\log^n x$ and on the other hand by increasing the exponent of the $\log x$ term. In order to prove the first part of this theorem, we use Proposition \ref{prop101}. For the second and third part, we need the following effective estimates for the Chebyshev $\vartheta$-function.

\begin{lem} \label{lem301}
For every $x \geq 1,091,159$, one has
\begin{equation}
|\vartheta(x) - x| \leq \frac{57.184x}{\log^4 x}, \tag{3.5} \label{3.5}
\end{equation}
and for every $x >1$, one has
\begin{equation}
|\vartheta(x) - x| \leq \frac{134,410x}{\log^5 x}. \tag{3.6} \label{3.6}
\end{equation}
\end{lem}

\begin{proof}
Using \cite[Table 15]{kadiri}, we see that the inequality \eqref{3.5} holds for every $x \geq 5 \times 10^6$. For smaller values of $x$, we use a computer. The inequality \eqref{3.6} has already been proven in \cite[Table 15]{kadiri}.
\end{proof}

Now we give a prove of Theorem \ref{thm102}.

\begin{proof}[Proof of Theorem \ref{thm102}]
For a better readability, we set $f_{k,n}(x) = kx/\log^nx$, $a = 0.048668 + 8.22 \times 10^{-11}$, $b = 114.368 + 2.145 \times 10^{-10}$, and $c = 268,820 + 5.0363 \times 10^{-7}$. 
Using Proposition \ref{prop101}, we get
\begin{displaymath}
\vartheta(x+f_{a, 3}(x)) - \vartheta(x) > \frac{x}{\log^3x} \left( 8.22 \times 10^{-11} - \frac{0.024334a}{\log^3x} \right) \geq 0
\end{displaymath}
for every $x \geq e^{243.3297}$, which implies that for every $x \geq e^{243.329}$ there is a prime number $p$ satisfying $x < p \leq x + ax/\log^3x$. From \eqref{3.2}, it is clear that the claim follows for every $x$ with $468,991,632 \leq x < e^{243.34}$. To deal with the cases where $17,051,887 \leq x < 468,991,632$, we check with a computer that the inequality $p_n(1 + a/\log^3 p_n) > p_{n+1}$ holds for every integer $n$ such that $\pi(17,051,887) \leq n \leq \pi(468,991,632)+1$. Finally, we notice that $\pi(x(1+ a/\log^3 x)) > \pi(x)$ for every $x$ such that $17,051,708 \leq x < 17,051,887$. 

In order to prove the second part, we use \eqref{3.5} to obtain the inequality
\begin{displaymath}
\vartheta(x+f_{b, 4}(x)) - \vartheta(x) > \frac{x}{\log^4x} \left( 2.145 \times 10^{-10} - \frac{57.184b}{\log^4x} \right) \geq 0
\end{displaymath}
for every $x \geq e^{2349.839}$. Obviously, the first part yields that there is a prime number $p$ satisfying $x < p \leq x + bx/\log^4x$ for every $17,051,708 
\leq x \leq e^{2349.963}$. Analogously to the proof of the first part, we check with a computer that for every $x$ with $2 \leq x < 17,051,708$ there is a prime $p$ so that $x < p \leq x + bx/\log^4x$.

Finally, we verify the third part. By \eqref{3.6}, we have
\begin{displaymath}
\vartheta(x+f_{c, 5}(x)) - \vartheta(x) > \frac{x}{\log^5x} \left( 5.0363 \times 10^{-7} - \frac{134410c}{\log^5x} \right) \geq 0
\end{displaymath}
for every $x \geq e^{2350.479}$. Now it suffices to observe that the second part implies the third part for every $x$ satisfying $2 \leq x \leq e^{2350.482}$. 
\end{proof}

\begin{rema}
Beginning with Hoheisel \cite{hoheisel}, many authors have found shorter intervals of the form $[x - x^{\delta}, x]$ that must contain a prime number for all sufficiently large values of $x$. The most recent result is due to Baker, Harman, and Pintz \cite{baker}. They found the value $\delta = 0.525$. Under the assumption that the Riemann hypothesis is true, much better results are known. For more details, see, for instance, Ramar\'e and Saouter \cite{ramare}, Dudek \cite{dudekRH}, Dudek, Greni\'e, and Molteni \cite{dudekal}, and Carneiro, Milinovich, and Soundararajan \cite{fourier}.
\end{rema}

\section{Proof of Theorem \ref{thm103}}

First, we note some well known estimates for the prime counting function $\pi(x)$. A classic method of finding explicit estimates for $\pi(x)$ is the following. Let $k$ be a positive integer and $\eta_k$ a positive real number. By \eqref{1.5}, there is a real number $x_1 = x_1(k, \eta_k) > 1$ so that
\begin{displaymath}
|\vartheta(x) - x| < \frac{\eta_kx}{\log^k x}
\end{displaymath}
for every $x \geq x_1$. In order to prove Theorem \ref{thm103}, we define the auxiliary function
\begin{equation}
J_{k;\eta_k;x_1}(x) = \pi(x_1) - \frac{\vartheta(x_1)}{\log x_1} + \frac{x}{\log x} + \frac{\eta_k x}{\log^{k+1} x} + \int_{x_1}^{x}{\left(
\frac{1}{\log^{2} t} + \frac{\eta_k}{\log^{k+2} t} \, \text{d}t \right)} \tag{4.1} \label{4.1}
\end{equation}
and note the following both inequalities involving the prime counting function $\pi(x)$.

\begin{lem} \label{lem401}
For every $x \geq x_1$, we have
\begin{displaymath}
J_{k;-\eta_k;x_1}(x) \leq \pi(x) \leq J_{k;\eta_k;x_1}(x).
\end{displaymath}
\end{lem}

\begin{proof}
The claim follows directly form \eqref{1.7} and \eqref{1.5}.
\end{proof}

One of the first estimates for $\pi(x)$ is due to Gauss. In 1793, he computed that
\begin{equation}
\pi(x) \leq \text{li}(x) \tag{4.2} \label{4.2}
\end{equation}
holds for every $x$ with $2 \leq x \leq 3,000,000$ and conjectured that the inequality \eqref{4.2} holds for every $x \geq 2$. This conjecture was disproven by Littlewood \cite{littlewood}. More precisely, he proved that the function $\pi(x) - \text{li}(x)$ changes the sign infinitely many times. Unfortunetely, Littlewood’s proof is nonconstructive and there is still no example of $x$ such that $\pi(x) > \text{li}(x)$. 
Skewes \cite{skewes} proved the existence of a number $x_0$ with $x_0 < \exp(\exp(\exp(\exp(7.705))))$
such that $\pi(x_0) > \text{li}(x_0)$. Lehman \cite{lehman} improved this last upper bound considerably by showing that exists a number $x_0$ with $x_0 < 1.65\times 10^{1165}$ such that $\pi(x_0) > \text{li}(x_0)$. After some further improvements (see, for instance, te Riele \cite{riele}, Bays and Hudson \cite{bayshudson}, Chao and Plymen \cite{chaoplymen}, Saouter and Demichel \cite{saouterdemichel}, 
Stoll and Demichel \cite{stolldemichel} and Saouter, Trudgian, and Demichel \cite{saoutertruddemichel}), the current best upper bound was found by Platt and Trudgian \cite{platttrudskewes}. They proved that there exists a number $x_0$ with $x_0 < e^{727.951332668}$ such that $\pi(x_0) > \text{li}(x_0)$. All these upper bounds have been proved by using computer calculations of zeros of the Riemann zeta function. The first lower bound for a number $x_0$ with $\pi(x_0) > \text{li}(x_0)$ was given by the calculation of Gauss, namely $x_0 > 3,000,000$. This lower bound was improved in a series of papers. For details, see Rosser and Schoenfeld \cite{rosserschoenfeld1962}, Brent \cite{brent}, Kotnik \cite{kotnik}, Platt and Trudgian \cite{platttrudskewes}, and Stoll and Demichel \cite{stolldemichel}. For our further inverstigation, we use the following improvement.

\begin{lem}[Büthe \cite{buthe2018}] \label{lem402}
For every $x$ with $2 \leq x \leq 10^{19}$, we have $\pi(x) \leq \emph{li}(x)$.
\end{lem}

\begin{rema}
Recently. Dusart \cite[Lemma 2.2]{dusart20182} showed that $\pi(x) \leq \text{li}(x)$ for every $x$ with $2 \leq x \leq 10^{20}$.
\end{rema}

Now we use Proposition \ref{prop101} and the Lemmata \ref{lem401} and \eqref{4.2} to give a proof of Theorem \ref{thm103}.

\begin{proof}[Proof of Theorem \ref{thm103}]
First, we combine Lemma \ref{lem401} with Proposition \ref{prop101} to see that
\begin{equation}
J_{3;-0.024334;x_1}(x) \leq \pi(x) \leq J_{3;0.024334;x_1}(x) \tag{4.3} \label{4.3}
\end{equation}
for every $x \geq x_1$, where $x_1 \geq 1,757,126,630,797$. Now, let $x_2 = 10^{18}$ and let $f(x)$ be given by the right-hand side of \eqref{1.12}. We consider the function $g(x) = f(x) - J_{3, 0.024334, x_2}(x)$. By \cite{deleglisetable}, we have $\vartheta(x_2) \geq 999,999,999,144,115,634$. Further, $\pi(x_2) = 24,739,954,287,740,860$ and so we compute $g(x_2) \geq 2 \times 10^8$. Since the derivative of $g$ is positive for every $x \geq x_2$, we get $f(x) - J_{3, 0.024334, x_2}(x) > 0$ for every $x \geq x_1$, and we conclude from \eqref{4.3} that the inequality \eqref{1.12} holds for every $x \geq x_1$. Comparing $f(x)$ with the integral logarithm $\text{li}(x)$, we see that $f(x) > \text{li}(x)$ for every $x \geq 121,141,948$. Now we can utilize Lemma \ref{lem402} to see that the desired inequality also holds for every $x$ such that $121,141,948 \leq x < 10^{18}$. 
A computer check for smaller values of $x$ completes the proof.
\end{proof}

Under the assumption that the Riemann hypothesis is true, von Koch \cite{koch1901} deduced that $\pi(x) = \text{li}(x) + O(\sqrt{x} \log x)$ as $x \to \infty$. Actually, one can show that the asymptotic formula $\pi(x) = \text{li}(x) + O(\sqrt{x} \log x)$ as $x \to \infty$ is even a sufficient criterion for the truth of the Riemann hypothesis. An explicit version of von Koch's result is due to Schoenfeld \cite[Corollary 1]{schoenfeld1976}. Under the assumption that the Riemann hypothesis is true, Schoenfeld found that the inequality
\begin{equation}
|\pi(x) - \text{li}(x)| < \frac{1}{8\pi}\, \sqrt{x} \log x \tag{4.4} \label{4.4}
\end{equation}
holds for every $x \geq 2,657$. In 2014, Büthe \cite[p.\:2,495]{buethe2016} proved that the inequality \eqref{4.4} holds unconditionally for every $x$ such that $2,657 \leq x \leq 1.4 \times 10^{25}$. Platt and Trudgian \cite[Corollary 1]{plattriemann} improved Büthe's result by showing that the inequality \eqref{4.4} holds unconditionally for every $x$ satisfying $2,657 \leq x \leq 2.169 \times 10^{25}$. Johnston \cite[Corollary 3.3]{johnston} extended the last result by showing the following

\begin{lem}[Johnston] \label{lem403}
The inequality \eqref{4.4} holds unconditionally for every $x$ satisfying $2,657 \leq x \leq 1.101 \times 10^{26}$.
\end{lem}

Now we can use Theorem \ref{thm103} and the Lemmata \ref{lem402} and \eqref{lem403} to find the following weaker but more compact upper bounds for the prime counting function $\pi(x)$ of the form
\begin{displaymath}
\pi(x) < \frac{x}{\log x - a_0 - \frac{a_1}{\log x} - \cdots - \frac{a_m}{\log^mx}} \q\q (x \geq x_0), \tag{4.5} \label{4.5}
\end{displaymath}
where $m$ is a integer with $0 \leq m \leq 5$ and $a_0, \ldots, a_m$ are suitable positive real numbers.

\begin{kor} \label{kor404}
We have
\begin{displaymath}
\pi(x) < \frac{x}{\log x - a_0 - \frac{a_1}{\log x} - \frac{a_2}{\log^2x}}
\end{displaymath}
for every $x \geq x_0$, where
\begin{center}
\begin{tabular}{|l||c|c|c|c|c|c|c|c|c|c|}
\hline
$a_0$\rule{0mm}{4mm} & $1.0343$ & $1$     & $1$    \\ \hline
$a_1$\rule{0mm}{4mm} & $0$      & $1.109$ & $1$    \\ \hline
$a_2$\rule{0mm}{4mm} & $0$      & $0$     & $3.48$ \\ \hline
$x_0$\rule{0mm}{4mm} & $106,640,139,304,611$    & $81,250,795,096,339$   & $145,413,088,724,077$ \\ \hline
\end{tabular} \ ,
\end{center}
and we have
\begin{displaymath}
\pi(x) < \frac{x}{\log x - 1 - \frac{1}{\log x} - \frac{3.024334}{\log^2x} - \frac{a_3}{\log^3x} - \frac{a_4}{\log^4x} - \frac{a_5}{\log^5x}}
\end{displaymath}
for every $x \geq x_0$, where
\begin{center}
\begin{tabular}{|l||c|c|c|c|c|c|c|c|c|c|}
\hline
$a_3$\rule{0mm}{4mm} & $14.893$ & $12.975666$ & $12.975666$ \\ \hline 
$a_4$\rule{0mm}{4mm} & $0$      & $79.962$    & $71.048668$ \\ \hline
$a_5$\rule{0mm}{4mm} & $0$      & $0$         & $533.594$ \\ \hline
$x_0$\rule{0mm}{4mm} & $142,464,507,937,911$   & $22$        & $32$\\ \hline
\end{tabular} \ .
\end{center}
\end{kor}

\begin{proof}
Theorem \ref{thm103} implies that the inequality
\begin{equation}
\pi(x) < \frac{x}{\log x - 1.0343} \tag{4.6} \label{4.6}
\end{equation}
holds for every $x \geq 108,943,258,198,427$. If we compare the right-hand side of \eqref{4.6} with $\text{li}(x)$, we can use Lemma \ref{lem402} to see that the required inequality \eqref{4.6} holds for every $x$ with $106,910,668,441,596 \leq x \leq 108,943,258,198,427$. Finally, we use Walisch's \textit{primecount} program \cite{walisch} to obtain that the inequality \eqref{4.6} is also valid for every $x$ satisfying $106,640,139,304,611 \leq x \leq 106,910,668,441,596$. The proof of each of the next three inequalities is similar to the proof of \eqref{4.6} and we leave the details to the reader. Next, we show that the inequality
\begin{equation}
\pi(x) < \frac{x}{\log x - 1 - \frac{1}{\log x} - \frac{3.024334}{\log^2x} - \frac{12.975666}{\log^3x} - \frac{79.962}{\log^4x}} \tag{4.7} \label{4.7}
\end{equation}
holds for every $x \geq 22$. First, we can use Theorem \ref{thm103} to obtain that the inequality \eqref{4.7} holds for every $x \geq 1.101 \times 10^{26}$. Let $f(x)$ denote the right-hand side of \eqref{4.6}. We get that $f(x) \geq \text{li}(x) + \sqrt{x}\log(x)/(8\pi)$ for every $x$ with $22,066,689,219,741,110 \leq x \leq 1.101 \times 10^{26}$. Now we can apply Lemma \ref{lem403} to see that the required inequality \eqref{4.7} also holds for every $x$ satisfying $22,066,689,219,741,110 \leq x \leq 1.101 \times 10^{26}$. A comparison with $\text{li}(x)$ shows that $f(x) > \text{li}(x)$ for every $x \geq 259,576,712,645$ and Lemma \ref{lem402} yields the desired inequality \eqref{4.7} for every $x$ with $259,576,712,645 \leq x \leq 22,066,689,219,741,110$. Finally, it suffices to apply Walisch's \textit{primecount} program \cite{walisch} to see that the inequality \eqref{4.7} also holds for every $x$ satisfying $22 \leq x \leq 259,576,712,645$. Again, the proof of the remaining inequality is similar to the proof of \eqref{4.7} and we leave the details to the reader.
\end{proof}

\begin{rema}
In Section 7, we give lots of other weaker upper bounds in the case where $m \in \{ 0,1,2\}$.
\end{rema}

Using Lemma \ref{lem301}, we get the following upper bound for the prime counting function which improves the inequality \eqref{1.12} for all sufficiently large values of $x$.

\begin{prop} \label{prop405}
For every $x \geq 29.53$, we have
\begin{displaymath}
\pi(x) < \frac{x}{\log x - 1 - \frac{1}{\log x} - \frac{3}{\log^2x} - \frac{70.935}{\log^3 x}}.
\end{displaymath}
\end{prop}

\begin{proof}
We combine Lemma \ref{lem401} with \eqref{3.5} to see that $\pi(x) \leq J_{4, 57.184, x_1}(x)$ for every $x \geq 10^{18}$ and proceed as in the proof of Theorem \ref{thm103}. We leave the details to the reader.
\end{proof}

Integration by parts in \eqref{1.8} implies that for every positive integer $m$, one has
\begin{equation}
\pi(x) = \frac{x}{\log x} + \frac{x}{\log^2 x} + \frac{2x}{\log^3 x} + \frac{6x}{\log^4 x} + \frac{24x}{\log^5 x} + \ldots + \frac{(m-1)! x}{\log^mx} + O \left( \frac{x}{\log^{m+1} x} \right) \tag{4.8} \label{4.8}
\end{equation}
as $x \to \infty$. In this direction, we get the following upper bound for $\pi(x)$.

\begin{prop} \label{prop406}
For every $x > 1$, we have
\begin{displaymath}
 \pi(x) < \frac{x}{\log x} + \frac{x}{\log^2 x} + \frac{2x}{\log^3 x} + \frac{6.024334x}{\log^4 x} + \frac{24.024334x}{\log^5 x} + \frac{120.12167x}{\log^6 x} + \frac{720.73002x}{\log^7 x} + \frac{6098x}{\log^8 x}.
\end{displaymath}
\end{prop}

\begin{proof}
We set $x_1 = 10^{18}$. Further, let $f(x)$ be the right-hand side of the required inequality. We have $f(x) > J_{3, 0.024334, x_1}(x)$ for every $x \geq x_1$. So, we can use \eqref{4.3} to get $f(x) > \pi(x)$ for every $x \geq x_1$. Since $f(x) > \text{li}(x)$ for every $x \geq 204,182,829$, we can apply Lemma \ref{lem402} to obtain $f(x) > \pi(x)$ for every $x$ such that $204,182,829 \leq x \leq x_1$. A direct computation for smaller values of $x$ completes the proof.
\end{proof}

Proposition \ref{prop406} yields the following weaker but more compact upper bounds for the prime counting function $\pi(x)$.

\begin{kor} \label{kor407}
For every $x \geq x_0$, we have
\begin{displaymath}
\pi(x) < \frac{x}{\log x} + \frac{x}{\log^2 x} + \frac{(2+ \e)x}{\log^3 x},
\end{displaymath} 
where
\begin{savenotes}
\begin{center}
\begin{tabular}{|l||c|c|c|c|c|c|}
\hline
$\e$\rule{0mm}{4mm}  & $0.21$ & $0.215$ & $0.22$ & $0.225$ \\ \hline
$x_0$\rule{0mm}{4mm} & $160,930,932,942,272$ & $83,016,503,500,865$ & $43,999,690,220,699$ & $23,824,649,646,672$ \\ \hline
\hline
$\e$\rule{0mm}{4mm}  & $0.23$ & $0.24$ & $0.25$ & $0.26$ \\ \hline
$x_0$\rule{0mm}{4mm} & $13,279,102,022,111$ & $4,511,700,549,332$ & $1,615,202,653,795$ & $643,809,266,445$\\ \hline
\hline
$\e$\rule{0mm}{4mm}  & $0.265 \footnote{This inequality was already known to be true for every $x \geq 8 \times 10^{11}$ (see \cite[Proposition 3.3]{sadegh}).}$ & $ 0.27$ & $ 0.28$ & $ 0.29$ \\ \hline
$x_0$\rule{0mm}{4mm} & $406,742,886,708$ & $265,248,130,170$ & $117,997,473,286$ & $57,720,805,589$\\ \hline
\end{tabular}.
\end{center}
\end{savenotes}
\end{kor}

\begin{proof}
Let $x_0 = 160,930,932,942,272$ and $f(x) = x/\log x + x/\log^2 x + 2.21x/\log^3x$. Proposition \ref{prop406} implies that $\pi(x) < f(x)$ for every $x \geq 180,250,881,352,396$. If we compare $f(x)$ with the integral logarithm $\text{li}(x)$, we get by Lemma \ref{lem402} that $\pi(x) < f(x)$ for every $x \geq 162,791,795,110,834$. Next, we use a computer to verify the inequality $\pi(x) < f(x)$ for every $x$ with $x_0 \leq x \leq 162,791,795,110,834$.
The remaining inequalities can be proved in the same way.
\end{proof}




\section{Proof of Theorem \ref{thm104}}

In order to give a proof of Theorem \ref{thm104}, we use \eqref{4.3} and a numerical calculation that verifies the desired inequality for smaller values of $x$.

%

\begin{proof}[Proof of Theorem \ref{thm104}]
Let $x_1 = 1,757,126,630,797$. Further, let $g(x)$ be the right-hand side of \eqref{1.13}.
We can compute that $J_{3, -0.024334, x_1}(x_1) - g(x_1) > 6 \times 10^3$. In addition we have $J_{3, -0.024334, x_1}'(x) > g'(x)$ for every $x \geq 44.42$. Therefore, we get $J_{3, -0.024334, x_1}(x) > g(x)$ for every $x \geq x_1$. Using \eqref{4.3}, we get the required inequality for every $x \geq x_1$. For smaller values of $x$ we use a computer.
\end{proof}

\begin{rema}
Let $x_1 = 1,751,189,194,177$. Then the inequality \eqref{1.13} does not hold for $x = x_1 - 0.1$.
\end{rema}

\begin{rema}
Theorem \ref{thm104} improves the lower bound for $\pi(x)$ obtained in \cite[Theorem 3]{axlerPrimeF}.
\end{rema}

In the next corollary, we establish some weaker lower bounds for the prime counting function.

\begin{kor} \label{kor501}
We have
\begin{displaymath}
\pi(x) > \frac{x}{\log x - 1 - \frac{1}{\log x} - \frac{a_2}{\log^2x} - \frac{a_3}{\log^3x} - \frac{a_4}{\log^4x} - \frac{a_5}{\log^5x}}
\end{displaymath}
for every $x \geq x_0$, where
\begin{center}
\begin{tabular}{|l||c|c|c|c|c|c|c|c|c|c|}
\hline
$a_2$\rule{0mm}{4mm} & $        2.975666$       & $        2.975666$ & $      2.975666$ & $  2.975666$ \\ \hline
$a_3$\rule{0mm}{4mm} & $       13.024334$       & $       13.024334$ & $     13.024334$ & $         0$ \\ \hline 
$a_4$\rule{0mm}{4mm} & $       70.951332$       & $       70.951332$ & $             0$ & $         0$ \\ \hline
$a_5$\rule{0mm}{4mm} & $      460.634397856444$ & $               0$ & $             0$ & $         0$ \\ \hline
$x_0$\rule{0mm}{4mm} & $ 1,035,745,443,241$     & $ 153,887,581,621$ & $ 7,713,187,213$ & $54,941,209$ \\ \hline
\end{tabular} \ .
\end{center}
\end{kor}

\begin{proof}
From Theorem \ref{thm104}, it follows that each required inequality holds for every $x \geq 1,751,189,194,177$. For smaller values of $x$ we use a computer.
\end{proof}

Further, we give the following result which refines Theorem \ref{thm104} for all sufficiently large values of $x$.

\begin{prop} \label{prop502}
For every $x \geq 467,497 = p_{39,021}$, we have
\begin{equation}
\pi(x) > \frac{x}{\log x - 1 - \frac{1}{\log x} - \frac{3}{\log^2x} + \frac{44.184}{\log^3 x}}. \tag{5.1} \label{5.1}
\end{equation}
\end{prop}

\begin{proof}
Let $x_1 = 10^7$ and let $f(x)$ denote the right-hand side of \eqref{5.1}. A comparison with $J_{4, -57.184, x_1}(x)$ gives that $J_{4, -57.184, x_1}(x) > f(x)$ for every $x \geq x_1$. Now we can use \eqref{3.5} and Lemma \ref{lem401} to see that $\pi(x) > f(x)$ for every $x \geq x_1$. We may conclude with a direct computation.
\end{proof}

The asymptotic expansion \eqref{1.11} implies that the slightly sharper inequality
\begin{equation}
\pi(x) > \frac{x}{\log x - 1 - \frac{1}{\log x} - \frac{3}{\log^2x}} \tag{5.2} \label{5.2}
\end{equation}
holds for all sufficiently large values of $x$. Under the assumption that the Riemann hypothesis is true, the present author \cite[Proposition 2]{axler20182} 
showed that the inequality \eqref{5.2} holds for every $x \geq 65,405,887$. Now we use Theorem \ref{thm104} to obtain the following unconditionally result.

\begin{prop} \label{prop503}
The inequality \eqref{5.2} holds unconditionally for every $x$ such that $65,405,887 \leq x \leq e^{1697}$ and every $x \geq e^{2256}$.
\end{prop}

\begin{proof}
In \cite[Theorem 1]{axler20182}, the inequality was already proved for every $x$ with $65,405,887 \leq x \leq 2.7358 \times 10^{40}$. If we utilize Theorem \ref{thm104}, it turns out that the inequality \eqref{5.2} holds unconditionally for every $x$ such that $65,405,887 \leq x \leq e^{540}$.

Now, let $f(x)$ denote the right-hand side of \eqref{5.2}. In order to verify the required inequality for every $x$ with $e^{540} \leq x \leq e^{1680}$, we set $c_0 = 1-1.6341 \times 10^{-12}$. By \cite[Table 3]{kadiri3}, we have $\vartheta(x) \geq c_0x$ for every $x > e^{500}$. Applying this inequality to \eqref{1.7}, we get
\begin{equation}
\pi(x) > g_0(x) \tag{5.3} \label{5.3}
\end{equation}
for every $x \geq e^{500}$, where $g_0(x) = c_0(\text{li}(x) - \text{li}(e^{500}) + e^{500}/500)$. If we show that $g_0(x) > f(x)$ for every $x$ satisfying $e^{540} \leq x \leq e^{1680}$, we can use \eqref{5.3} to see that the required inequality \eqref{5.2} holds for every $x$ with $e^{540} \leq x \leq e^{1680}$. Since $g_0'(x) > f'(x)$ for every $x$ so that $9 \leq x \leq e^{1680}$, it remains to show that $g_0(x_0) > f(x_0)$, where $x_0 = e^{540}$. First, we note that
\begin{equation}
\sum_{k=1}^6 \frac{(k-1)!}{\log^kt} < \frac{\text{li}(t)}{t} < \frac{1.003}{\log t}, \tag{5.4} \label{5.4}
\end{equation}
where the left-hand side inequality holds for every $t \geq 565$ and the right-hand side inequality is valid for every $t \geq e^{500}$.
Therefore,
\begin{displaymath}
\frac{g_0(x_0) - f(x_0)}{x_0} > c_0 \sum_{k=1}^6 \frac{(k-1)!}{540^k} - \frac{0.003 c_0}{e^{40}} - \frac{f(x_0)}{x_0}.
\end{displaymath}
Since the right-hand side of the last inequality is positive and we conclude that the required inequality holds for every $x$ with $x_0 \leq x \leq e^{1680}$.

Next, we check the inequality \eqref{5.2} for every $x$ satisfying $e^{1680} \leq x \leq e^{1697}$. Here, we set $c_1 = 1 - 1.5733 \times 10^{-12}$ and $d_1 = 1 - 1.5907 \times 10^{-12}$. According to Fiori, Kadiri, and Swidinsky \cite[Table 3]{kadiri3}, we have
\begin{align}
\vartheta(x) \geq c_1x \q\q (x \geq e^{1680}), \tag{5.5} \label{5.5}\\
\vartheta(x) \geq d_1x \q\q (x \geq e^{1000}). \tag{5.6} \label{5.6}
\end{align}
If we substitute the inequalities \eqref{5.5} and \eqref{5.6} into \eqref{1.7}, we get that
\begin{equation}
\pi(x) > g_{1680}(x) \tag{5.7} \label{5.7}
\end{equation}
for every $x \geq e^{1680}$, where
\begin{equation}
g_a(x) = c_1\left( \text{li}(x) - \text{li}(e^a) + \frac{e^a}{a} \right) + d_1 \left( \text{li}(e^a) - \frac{e^a}{a} - \text{li}(e^{1000}) + \frac{e^{1000}}{1000} \right). \tag{5.8} \label{5.8}
\end{equation}
Again it suffices to show that $g_{1680}(x) \geq f(x)$ for every $x$ satisfying $e^{1680} \leq x \leq e^{1697}$. We can use the left-hand side inequality of \eqref{5.4} and the inequality $\text{li}(t) < 1.0006t/\log t$, which holds for every $t \geq e^{1680}$, to see that
\begin{equation}
\frac{g_{1680}(e^{1680}) - f(e^{1680})}{e^{1680}} > \frac{c_1}{1680} + d_1\sum_{k=2}^6 \frac{(k-1)!}{1680^k} - \frac{0.0006d_1}{1000 e^{680}} - \frac{f(e^{1680})}{e^{1680}} > 0. \tag{5.9} \label{5.9}
\end{equation}
Together with $g_{1680}'(t) \geq f'(t)$ for every $t$ with $e^{1680} \leq t \leq t_1$, where $t_1 = 1696.0578605\ldots$, it turns out that $\pi(x) > g_{1680}(x) > f(x)$ for every $x$ with $e^{1680} \leq x \leq t_1$. Similar to \eqref{5.9}, we compute that $g_{1680}(e^{1697}) > f(e^{1697})$. Since $g_{1680}'(t) \leq f'(t)$ for every $t$ with $t \geq t_1$, we see that $\pi(x) > g_{1680}(x) > f(x)$ for every $x$ with $t_1 \leq x \leq e^{1697}$.

Now, we deal with the case where $x$ satisfies $e^{2256} \leq x \leq e^{2259}$. Here, let $c_2 = 1 - 5.0057 \times 10^{-13}$. By \cite[Table 3]{kadiri4}, we have $\vartheta(x) \geq c_2x$ for every $x \geq e^{2256}$. Similar to \eqref{5.7}, we get that $\pi(x) > g_{2256}(x)$ for every $x \geq e^{2256}$, where $g_a(x)$ is defined as in \eqref{5.8}. Analogous to the proof that $\pi(x) > g_{1680}(x) > f(x)$ for every $x$ with $e^{1680} \leq x \leq t_1$, we see that $\pi(x) > g_{2256}(x) > f(x)$ for every $x$ satisfying $e^{2256} \leq x \leq e^{2258}$.

The cases where
\begin{itemize}
\item $x$ satisfies $e^{2258} \leq x \leq e^{2265}$,
\item $x$ satisfies $e^{2265} \leq x \leq e^{2289}$,
\item $x$ satisfies $e^{2289} \leq x \leq e^{2377}$,
\item $x$ satisfies $e^{2377} \leq x \leq e^{4677}$
\end{itemize}
can be treated as the case where $x$ satisfies $e^{2256} \leq x \leq e^{2258}$ and we leave the details to the reader.

The final step of the proof consists in the verification of the required inequality for every $x \geq x_1$, where $x_1 = e^{4677}$. By \cite[Table 15]{kadiri}, we have $\vartheta(x) > x(1 - 0.037436/\log^4 x)$ for every $x \geq x_1$. Now we can utilize Lemma \ref{lem401} to get $\pi(x) \geq J_{4, -0.037436, x_1}(x)$ for every $x \geq x_1$. We want to show that $J_{4, -0.037436, x_1}(x) > f(x)$ for every $x \geq x_1$. Since $J_{4, -0.037436, x_1}'(x) > f'(x)$ for every $x \geq x_1$, it remains to show that $J_{4, -0.037436, x_1}(x_1) > f(x_1)$. For a better readability, we set $J(x) =J_{4, -0.037436, x_1}(x)$.
By \cite[Table 3]{kadiri4}, we have $\vartheta(x) > c_3x$ for every $x \geq e^{2000}$, where $c_3 = 1-1.5692 \times 10^{-12}$. Applying this inequality to \eqref{1.7}, we see that
\begin{displaymath}
\pi(x_1) - \frac{\vartheta(x_1)}{\log x_1} \geq c_3 \left( \text{li}(x_1) - \frac{x_1}{\log x_1} - \text{li}(e^{2000}) + \frac{e^{2000}}{2000}\right).
\end{displaymath}
If we substitute this into \eqref{4.1}, we get
\begin{displaymath}
\frac{J(x_1) - f(x_1)}{x_1} > \frac{c_3}{x_1} \left( \text{li}(x_1) - \text{li}(e^{2000}) + \frac{e^{2000}}{2000} \right) + \frac{1.5692 \times 10^{-12}}{4677} - \frac{0.037436}{4677^5} - \frac{f(x_1)}{x_1}.
\end{displaymath}
Finally, we use \eqref{5.4} to see that
\begin{displaymath}
\frac{J(x_1) - f(x_1)}{x_1} > \frac{1}{4677} + c_3 \sum_{k=2}^6 \frac{(k-1)!}{4677^k} - \frac{c_3}{e^{2677}} + \frac{c_3}{2000e^{2677}} - \frac{0.037436}{4677^5} - \frac{f(x_1)}{x_1}.
\end{displaymath}
Since the right-hand side of the last inequality is positive, we obtain that that the required inequality \eqref{5.2} holds for every $x \geq e^{4677}$, and we arrive at the end of the proof.
\end{proof}

\begin{rema}
By \eqref{1.11}, the even sharper inequality
\begin{displaymath}
\pi(x) > \frac{x}{\log x - 1 - \frac{1}{\log x} - \frac{3}{\log^2x} - \frac{13}{\log^3x}}
\end{displaymath}
holds for all sufficiently large values of $x$. Similar to the proof of Proposition \ref{prop503}, we get that this inequality holds for every $x$ satisfying $11,471,757,461 \leq x \leq e^{57.820987}$ and every $x \geq e^{5000}$.
\end{rema}

Let $n$ be a positive integer. Then \eqref{4.8} yields the inequality
\begin{displaymath}
\pi(x) > \frac{x}{\log x} + \frac{x}{\log^2 x} + \frac{2x}{\log^3 x} + \frac{6x}{\log^4 x} + \frac{24x}{\log^5 x} + \ldots + \frac{(n-1)! x}{\log^nx} \tag{5.10} \label{5.10}
\end{displaymath}
for all sufficiently large values of $x$. In the following proposition, we describe a method to find lower bounds for $\pi(x)$ in the direction of \eqref{5.10} by using lower bounds for $\pi(x)$ in the direction of \eqref{1.11}.

\begin{prop} \label{prop504}
Let $n$ be a positive integer and let $a_0 > 0$ and $a_1, \ldots, a_n$ be negative real numbers. Suppose that there is a positive real number $x_0$ such that the inequalities
\begin{equation}
a_0\log x + a_1 + \frac{a_2}{\log x} + \ldots + \frac{a_n}{\log^{n-1}x} > 0 \tag{5.11} \label{5.11}
\end{equation}
and
\begin{equation}
\pi(x) > \frac{x}{ a_0\log x + a_1 + \frac{a_2}{\log x} + \ldots + \frac{a_n}{\log^{n-1}x}} \tag{5.12} \label{5.12}
\end{equation}
hold simultaneously for every $x \geq x_0$. Then we have
\begin{displaymath}
\pi(x) > \frac{b_0x}{\log x} + \frac{b_1x}{\log^2 x} + \ldots + \frac{b_nx}{\log^{n+1} x}
\end{displaymath}
for every $x \geq x_0$, where $b_0, \ldots, b_n$ are real numbers recursively defined by
\begin{equation}
b_0 = 1/a_0, \q\q \text{and} \q\q b_k = -\frac{1}{a_0} \sum_{i=1}^k a_i b_{k-1} \q (1 \leq k \leq n). \tag{5.13} \label{5.13}
\end{equation}
\end{prop}

\begin{proof}
For $y > 0$, we define $R(y) = \sum_{k=0}^n a_i/y^i$ and $S(y) = \sum_{i=0}^n b_i/y^i$. For $i \in \{ 1, \ldots, 2n\}$, we set
\begin{displaymath}
a'_i =
\begin{cases}
a_i, &\text {if $i \in \{ 1 , \ldots, n\}$,} \\
0, &\text {otherwise} \nonumber
\end{cases}
\quad\quad \text{and} \quad\quad
b'_i =
\begin{cases}
b_i, &\text {if $i \in \{ 1 , \ldots, n\}$,} \\
0, &\text {otherwise.} \nonumber
\end{cases}
\end{displaymath}
Using \eqref{5.13} together with $b'_{n+1} = 0$, we can see that
\begin{displaymath}
R(y) S(y) = 1 + \sum_{k=n+1}^{2n} \sum_{i=1}^k \frac{a'_ib'_{k-i}}{y^k}.
\end{displaymath}
Since $a'_ib'_{k-i} \leq 0$ for every $i$ with $1 \leq i \leq 2n$ and every $k$ satisfying $n+1 \leq k \leq 2n$, we get $R(y) S(y) \leq 1$. By \eqref{5.11}, we have $R(\log x) > 0 $ for every $x \geq x_0$. Now we can use \eqref{5.12} to get $\pi(x) > x/(R(x) \log x) \geq xS(\log x)/\log x$ for every $x \geq x_0$ and we arrive at the end of the proof.
\end{proof}

The best explicit result in the direction of \eqref{5.3} was found in \cite[Proposition 5]{axlerPrimeF}.
The following refinements of it are a consequence of Proposition \ref{prop504}, Theorem \ref{thm104}, and Corollary \ref{kor501}.

\begin{kor} \label{kor505}
We have
\begin{displaymath}
\pi(x) > \frac{x}{\log x} + \frac{x}{\log^2 x} + \frac{2x}{\log^3 x} + \frac{b_4x}{\log^4 x} + \frac{b_5x}{\log^5 x} + \frac{b_6x}{\log^6 x} + 
\frac{b_7x}{\log^7 x} + \frac{b_8x}{\log^8 x}
\end{displaymath}
for every $x \geq x_0$, where
\begin{center}
\begin{tabular}{|l||c|c|c|c|c|c|c|c|c|c|}
\hline
$b_4$\rule{0mm}{4mm} & $          5.975666$ & $        5.975666$ & $        5.975666$ & $      5.975666$ & $   5.975666$\\ \hline
$b_5$\rule{0mm}{4mm} & $         23.975666$ & $       23.975666$ & $       23.975666$ & $     23.975666$ & $          0$\\ \hline 
$b_6$\rule{0mm}{4mm} & $         119.87833$ & $       119.87833$ & $      119.87833 $ & $             0$ & $          0$\\ \hline
$b_7$\rule{0mm}{4mm} & $         719.26998$ & $       719.26998$ & $               0$ & $             0$ & $          0$\\ \hline
$b_7$\rule{0mm}{4mm} & $        5034.88986$ & $               0$ & $               0$ & $             0$ & $          0$\\ \hline
$x_0$\rule{0mm}{4mm} & $ 1,681,111,802,141$ & $ 721,733,241,667$ & $ 110,838,719,141$ & $ 1,331,691,853$ & $ 10,383,799$ \\ \hline
\end{tabular} \ .
\end{center}
\end{kor}

\begin{proof}
In order to prove the first inequality, we combine Proposition \ref{prop504} and Theorem \ref{thm104} to see that this inequality holds for every $x \geq 1,751,189,194,177$. For smaller values of $x$, we use a computer. Further, we use Proposition \ref{prop504}, Corollary \ref{kor501}, and a direct computation for smaller values of $x$ to verify the remaining inequalities.
\end{proof}

\begin{rema}
By \eqref{5.10}, we see that the inequality
\begin{equation}
\pi(x) > \frac{x}{\log x} + \frac{x}{\log^2 x} + \frac{2x}{\log^3 x} + \frac{6x}{\log^4 x} \tag{5.14} \label{5.14}
\end{equation}
holds for all sufficiently large values of $x$. If we combine Proposition \ref{prop504}, Proposition \ref{prop503}, and \cite[Theorem 2]{axler20182}, it turns out that the inequality \eqref{5.14} holds for every $x$ such that $10,384,261 \leq x \leq e^{1697}$ and every $x \geq e^{2256}$.
\end{rema}

\section{Proof of Theorem \ref{thm105}}

In this section, we want to find unrestricted effective estimates for the sum
\begin{displaymath}
\sum_{p \leq x} \frac{1}{p}
\end{displaymath}
where $p$ runs over primes not exceeding $x$.
%
%
%
For this purpose, we use the method investigated by Rosser and Schoenfeld \cite[p.\:74]{rosserschoenfeld1962}. They derived a remarkable identity which connects the sum of the reciprocals of all prime numbers not exceeding $x$ with Chebyshev's $\vartheta$-function by showing that
\begin{equation}
A_1(x) =  \frac{\vartheta(x) - x}{x \log x} - \int_x^{\infty} \frac{(\vartheta(y)-y)(1 + \log y)}{y^2\log^2y} \, \text{d}y, \tag{6.1} \label{6.1}
\end{equation}
where
\begin{equation}
A_1(x) = \sum_{p \leq x} \frac{1}{p} - \log \log x - B. \tag{6.2} \label{6.2}
\end{equation}
Here, the constant $B$ is defined as in \eqref{1.15}. Applying \eqref{1.2} to \eqref{6.1}, Rosser and Schoenfeld \cite[p.\:68]{rosserschoenfeld1962} refined the error term in Mertens' result \eqref{1.14} by giving $A_1(x) =  O (\exp(-a\sqrt{\log x}))$ as $x \to \infty$, where $a$ is an absolute positive constant. Then \cite[Theorem 5]{rosserschoenfeld1962} they used explicit estimates for Chebyshev's $\vartheta$-function to show that
\begin{displaymath}
- \frac{1}{2 \log^2 x} < A_1(x) < \frac{1}{2 \log^2 x}, \tag{6.3} \label{6.3}
\end{displaymath}
where the left-hand side inequality is valid for every $x > 1$ and the right-hand side inequality holds for every $x \geq 286$. Meanwhile there are several improvements of \eqref{6.3} (see, for instance, \cite[Theorem 5.6]{dusart20181} and \cite[Proposition 7]{axlerPrimeF}). In Theorem \ref{thm105}, we give the current best unconditionally effective estimates for $A_1(x)$. The proof is now rather simple.

\begin{proof}[Proof of Theorem \ref{thm105}]
It suffices to combine \eqref{6.1} with Proposition \ref{prop101}.
\end{proof}

\begin{rema}
Note that the positive integer $N_0 = 1,757,126,630,797$ might not be the smallest positive integer $N$ so that the inequality given in Theorem  \ref{thm105} holds for every $x \geq N$.
\end{rema}

\begin{rema}
 Rosser and Schoenfeld \cite[Theorem 20]{rosserschoenfeld1962} used the calculation in \cite{appel} to see that $A_1(x) > 0$ for every $1 < x \leq 10^8$ and raised the question whether this inequality hold for every $x > 1$. Robin \cite[Th\'eor\`eme 2]{robin1983} proved that the function $A_1(x)$ changes the sign infinitely often, which leads to a negative answer to the obove question. By adapting a method for bounding Skewes' number, Büthe \cite[Theorem 1.1]{buethe2015} found that there exists an $x_0 \in [\exp(495.702833109),\exp(495.702833165)]$ such that $A_1(x)$ is negative for every $x \in [x_0 - \exp(239.046541), x_0]$.
\end{rema}

\begin{rema}
Under the assumption that the Riemann hypothesis is true, Schoenfeld \cite[Corollary 2]{schoenfeld1976} found some better estimate for the sum of the reciprocals of all prime numbers not exceeding $x$.
This result was recently improved by Dusart \cite[Theorem 4.1]{dusart20182}.
\end{rema}

Using the definition \eqref{1.15} of $B$, we get
\begin{equation}
e^{\gamma} \log x \prod_{p \leq x} \left( 1 - \frac{1}{p} \right) = e^{-S(x) - A_1(x)}, \tag{6.4} \label{6.4}
\end{equation}
where
\begin{equation}
S(x) = \sum_{p>x} \left( \log \left( 1 - \frac{1}{p} \right) + \frac{1}{p} \right) = - \sum_{n=2}^{\infty} \frac{1}{n} \sum_{p > x} \frac{1}{p^n}. \tag{6.5} \label{6.5}
\end{equation}
By Rosser and Schoenfeld \cite[p.\:87]{rosserschoenfeld1962}, we have
\begin{equation}
- \frac{1.02}{(x-1)\log x} < S(x) < 0 \tag{6.6} \label{6.6}
\end{equation}
for every $x > 1$. Hence,
the asymptotic formula \eqref{1.14} gives $A_2(x) = O(1/\log^2x)$ as $x \to \infty$, where
\begin{displaymath}
A_2(x) = \frac{e^{-\gamma}}{\log x} - \prod_{p \leq x} \left( 1 - \frac{1}{p} \right).
\end{displaymath}
In \cite[Theorem 7]{rosserschoenfeld1962}, Rosser and Schoenfeld found that
\begin{displaymath}
\frac{e^{-\gamma}}{\log x} \left( 1 - \frac{1}{2 \log^2 x} \right) < \prod_{p \leq x} \left( 1 - \frac{1}{p} \right) < \frac{e^{-\gamma}}{\log x} \left( 1 + \frac{1}{2 \log^2 x} \right),
\end{displaymath}
where the left-hand side inequality is valid for every $x \geq 285$ and the right-hand side inequality holds for every $x > 1$.
We use \eqref{6.4} combined with Theorem \ref{thm105} to obtain the following refinement of \cite[Proposition 9]{axlerPrimeF}.

\begin{prop} \label{prop601}
For every $x \geq 1,757,126,630,797$, we have
\begin{displaymath}
\frac{e^{-\gamma}}{\log x} \exp(-f(x)) < \prod_{p \leq x} \left( 1 - \frac{1}{p} \right) < \frac{e^{-\gamma}}{\log x} \exp \left( f(x) + \frac{1.02}{(x-1)\log x} \right), 
\end{displaymath}
where $f(x)$ denotes the right-hand side of \eqref{1.16}.
\end{prop}

\begin{proof}
First, we apply the left-hand side inequality of Theorem \ref{thm105} to \eqref{6.4} and see that
\begin{equation}
\prod_{p\leq x} \left( 1 - \frac{1}{p} \right) < \frac{e^{-\gamma}}{\log x} \exp( - S(x) + f(x)) \tag{6.7} \label{6.7}
\end{equation}
for every $x > 1,757,126,630,797$. Now it suffices to apply the right-hand side inequality of \eqref{6.6} to \eqref{6.7} and we get the required right-hand side inequality. One the other hand, we have $S(x) < 0$ by \eqref{6.6}. Applying this and the right-hand side inequality of Theorem \ref{thm105} to \eqref{6.4}, we arrive at the end of the proof.
%
\end{proof}

\begin{rema}
Note that the positive integer $N_0 = 1,757,126,630,797$ in Proposition \ref{prop601} might not be the smallest positive integer $N$ so that the inequality given  holds for every $x \geq N$.
\end{rema}

\begin{rema}
Under the assumption that the Riemann hypothesis is true, Schoenfeld \cite[Corollary 3]{schoenfeld1976} found that the inequality
\begin{displaymath}
|A_2(x)| < \frac{3 \log x + 5}{8 \pi e^{\gamma}\sqrt{x} \log x}
\end{displaymath}
holds for every $x \geq 8$. This was slightly improved by Dusart \cite[Theorem 4.4]{dusart20182} in 2018.
\end{rema}


\begin{rema}
Rosser and Schoenfeld \cite[Theorem 23]{rosserschoenfeld1962} found that $A_2(x) > 0$ for every $0 < x \leq 10^8$ and stated \cite[p.\:73]{rosserschoenfeld1962} the question whether this inequality also hold for every $x > 10^8$. In \cite[Proposition 1]{robin1983}, Robin answered this by showing that the function $A_2(x)$ changes the sign infinitely often.
\end{rema}


Now we can use Proposition \ref{prop601} to derive the following effective estimates for
\begin{displaymath}
\prod_{p \leq x} \left( 1 + \frac{1}{p} \right),
\end{displaymath}
where $p$ runs over primes not exceeding $x$.

\begin{kor} \label{kor602}
For every $x \geq 1,757,126,630,797$, one has
\begin{displaymath}
\frac{6e^{\gamma}}{\pi^2} \exp \left( -f(x) - \frac{1.02}{(x-1)\log x} \right) \log x < \prod_{p \leq x} \left( 1 + \frac{1}{p} \right) < \frac{6e^{\gamma}}{\pi^2} \left( 1 + \frac{1}{x} \right) \exp(f(x))\log x, 
\end{displaymath}
where $f(x)$ denotes the right-hand side of \eqref{1.16}.
\end{kor}

\begin{proof}
Since $1+1/p = (1-1/p^2)/(1-1/p)$, we can use Proposition \ref{prop601} and \cite[Lemma 4.3]{dusart20182} to get that
\begin{displaymath}
\frac{e^{\gamma}}{\zeta(2)}  \exp \left(-f(x) - \frac{1.02}{(x-1)\log x} \right)\log x < \prod_{p \leq x} \left( 1 + \frac{1}{p} \right) < \frac{e^{\gamma}}{\zeta(2)} \left( 1 + \frac{1}{x} \right) \exp(f(x))\log x
\end{displaymath}
for every $x \geq 1,757,126,630,797$. Finally, it suffices to apply the well known identity $\zeta(2) = \pi^2/6$.
\end{proof}

\begin{rema}
Note that the positive integer $N_0 = 1,757,126,630,797$ might not be the smallest positive integer $N$ so that the inequality given in Corollary \ref{kor602} holds for every $x \geq N$.
\end{rema}

Let us briefly study $S(x)$, defined as in \eqref{6.5}, in more detail. In the proof of the left-hand side inequality in \eqref{6.6}, Rosser and Schoenfeld used the inequality $\vartheta(x) < 1.02x$ which is valid for every $x > 0$ (see \cite[Theorem 9]{rosserschoenfeld1962}). If we use approximations for $\vartheta(x)$ of the form \eqref{1.5}, we get the following result.

\begin{prop} \label{prop603}
Let $k$ be a positive integer and let $\eta_k$ and $x_0 = x_0(k)$ be positive real numbers with $x_0 > 1$ so that $|\vartheta(x)-x| <\eta_kx/\log^kx$ for every $x \geq x_0$. Then, we have
\begin{displaymath}
\left| S(x) - \sum_{n=1}^\infty \frac{\emph{li}(x^{-n})}{n+1} \right| < \frac{\eta_k}{\log^{k+1}x} \left( (x+1) \log \left( \frac{x}{x-1} \right) - 1 \right)
\end{displaymath}
for every $x \geq x_0$.
\end{prop}

In order to prove this proposition, we first establish the following lemma.

\begin{lem} \label{lem604}
Let $n$ be a positive integer with $n\geq 2$. Under the assumptions of Proposition \ref{prop603}, we have
\begin{displaymath}
\left| \emph{li}(x^{1-n}) + \sum_{p > x} \frac{1}{p^n} \right| < \frac{\eta_k}{x^{n-1}\log^{k+1}x} \left( 1 + \frac{n}{n-1} \right)
\end{displaymath}
for every $x \geq x_0$.
\end{lem}

\begin{proof}
By \cite[p.\:87]{rosserschoenfeld1962}, we have
\begin{equation}
\sum_{p > x} \frac{1}{p^n} = - \frac{\vartheta(x)}{x^n\log x} + \int_x^{\infty} \frac{(1+n\log y) \vartheta(y)}{y^{n+1} \log^2 y} \, \text{d}y. \tag{6.8} \label{6.8}
\end{equation}
Since we have assumed that $|\vartheta(x)-x| <\eta_kx/\log^kx$ for every $x \geq x_0$, we see that
\begin{equation}
\sum_{p > x} \frac{1}{p^n} \leq - \text{li}(x^{1-n}) + \frac{\eta_k}{x^{n-1}\log^{k+1}x} + \eta_k \int_x^{\infty} \frac{1+n\log y}{y^n \log^{k+2} y} \, \text{d}y \tag{6.9} \label{6.9}
\end{equation}
for every $x \geq x_0$. Analogous to \cite[Lemma 9]{rosserschoenfeld1962}, we get that
\begin{displaymath}
\int_x^{\infty} \frac{1+n\log y}{y^n \log^{k+2} y} \, \text{d}y \leq \frac{n}{(n-1)x^{n-1}\log^{k+1}x}.
\end{displaymath}
Applying this inequality to \eqref{6.9}, we see that the required upper bound holds for every $x \geq x_0$. The proof of the required lower bound is quite similar and we leave the details to the reader.
\end{proof}

Now we can combine the definition \eqref{6.5} with Lemma \ref{lem604} to get the following proof of Proposition \ref{prop603}.

\begin{proof}[Proof of Proposition \ref{prop603}]
If we apply Lemma \ref{lem604} to \eqref{6.5}, it turns out that
\begin{displaymath}
-\frac{\eta_k}{\log^{k+1}x} \sum_{n=2}^\infty \left( 1 + \frac{n}{n-1} \right) \frac{1}{nx^{n-1}} < S(x) - \sum_{n=1}^\infty \frac{\text{li}(x^{-n})}{n+1} <  \frac{\eta_k}{\log^{k+1}x} \sum_{n=2}^\infty \left( 1 + \frac{n}{n-1} \right) \frac{1}{nx^{n-1}}
\end{displaymath}
for every $x \geq x_0$. Now, it suffices to apply the identity
\begin{displaymath}
\sum_{n=2}^\infty \left( 1 + \frac{n}{n-1} \right) \frac{1}{nx^{n-1}} = (x+1) \log \left( \frac{x}{x-1} \right) - 1
\end{displaymath}
to complete the proof.
%
%
%
\end{proof}

If we combine \eqref{2.12} and \eqref{6.8}, we find the following new necessary condition for the Riemann hypothesis including the sum in Lemma \ref{lem604}.


\begin{prop} \label{prop605}
Let $n$ be a positive integer with $n\geq 2$. Under the assumption that the Riemann hypothesis is true, we have
\begin{displaymath}
\left| \emph{li}(x^{1-n}) + \sum_{p > x} \frac{1}{p^n} \right| < \frac{1}{8\pi x^{n-1/2}} \left( 1 + \frac{2n}{2n-1} \right)\left( \log x + \frac{2}{2n-1} \right)
\end{displaymath}
for every $x \geq 599$.
\end{prop}

\begin{proof}
Instead of the assumption \eqref{1.5}, we now use \eqref{2.12} in the proof of Lemma \ref{lem604}.
\end{proof}


%

\section{Proof of Theorem \ref{thm106}}


%
%
%

Here we give the following proof of Theorem \ref{thm106}.

\begin{proof}[Proof of Theorem \ref{thm106}]
Let the constant $E$ be defined as in \eqref{1.18} and let 
\begin{equation}
A_3(x) = \sum_{p \leq x} \frac{\log p}{p} - \log x - E. \tag{7.1} \label{7.1}
\end{equation}
By Rosser and Schoenfeld \cite[p.\:74]{rosserschoenfeld1962}, we have
\begin{equation}
A_3(x) =  \frac{\vartheta(x) - x}{x} - \int_x^{\infty} \frac{\vartheta(y) - y}{y^2} \, \text{d}y. \tag{7.2} \label{7.2}
\end{equation}
Similarly to the proof of Theorem \ref{thm105}, we may combine \eqref{7.2} and Proposition \ref{prop101} to conclude that the desired both inequalities hold for every $x \geq 1,757,126,630,797$.
\end{proof}

\begin{rema}
Note that the positive integer $N_0 = 1,757,126,630,797$ in Theorem \ref{thm106} might not be the smallest positive integer $N$ so that the inequality given in Theorem \ref{thm106} holds for every $x \geq N$.
\end{rema}

\begin{rema}
Under the assumption that the Riemann hypothesis is true, Schoenfeld \cite[Corollary 2]{schoenfeld1976} found a better upper bound for $|A_3(x)|$. This result was later improved by Dusart \cite[Theorem 4.2]{dusart20182}.
\end{rema}

\begin{rema}
Rosser and Schoenfeld \cite[Theorem 21]{rosserschoenfeld1962} also found that $A_3(x) > 0$ for every $0 < x \leq 10^8$. Again, they asked whether this inequality also holds for every $x > 10^8$. Robin \cite[Proposition 1]{robin1983} showed that the function $A_3(x)$ changes the sign infinitely often, which leads again to a negative answer to the above question. Unfortunately, until today no $x_0$ is known so that $A_3(x_0) < 0$.
\end{rema}

\section{Appendix}

In this section we use Corollary \ref{kor404} and Walisch's \textsl{primecount} program \cite{walisch} to note more weaker upper bounds for the prime counting function $\pi(x)$ of the form \eqref{4.5},
where $m$ is an integer with $0 \leq m \leq 2$ and $a_0, \ldots, a_m$ are suitable positive real numbers. We start with the case where $m=0$.

\begin{kor} \label{kor801}
One has
\begin{displaymath}
\pi(x) < \frac{x}{\log x - a_0}
\end{displaymath}
for every $x \geq x_0$, where
\begin{center}
\begin{tabular}{|l||c|c|c|c|c|}
\hline
$a_0$\rule{0mm}{4mm} & $1.0344$ & $1.0345$ & $1.0346$ & $1.0347$ \\ \hline
$x_0$\rule{0mm}{4mm} & $98,011,218,006,714$ & $90,093,726,828,053$ & $82,972,765,680,514$ & $76,292,362,570,940$\\ \hline
\hline
$a_0$\rule{0mm}{4mm} & $1.0348$ & $1.0349$ & $1.035$ & $1.036$ \\ \hline
$x_0$\rule{0mm}{4mm} & $70,363,470,737,452$ & $64,716,191,738,353$ & $59,667,044,596,151$ & $27,086,141,056,455$ \\ \hline
\hline
$a_0$\rule{0mm}{4mm} & $1.037$ & $1.038$ & $1.039$ & $1.04$ \\ \hline
$x_0$\rule{0mm}{4mm} & $12,806,615,320,917$ & $6,317,261,904,937$ & $3,231,501,496,562$ & $1,697,021,254,855$ \\ \hline
\hline
$a_0$\rule{0mm}{4mm} & $1.041$ & $1.042$ & $1.043$ & $1.044$ \\ \hline
$x_0$\rule{0mm}{4mm} & $924,640,658,874$ & $519,205,451,664$ & $296,735,291,225$ & $175,758,684,156$ \\ \hline
\hline
$a_0$\rule{0mm}{4mm} & $1.045$ & $1.046$ & $1.047$ & $1.048$ \\ \hline
$x_0$\rule{0mm}{4mm} & $105,640,136,371$ & $65,431,161,562$ & $41,022,022,044$ & $25,724,702,310$ \\ \hline
\hline
$a_0$\rule{0mm}{4mm} & $1.049$ & $1.05$ & $1.051$ & $1.052$ \\ \hline
$x_0$\rule{0mm}{4mm} & $17,231,171,472$ & $11,207,440,881$ & $7,538,561,672$ & $5,047,295,951$ \\ \hline
\hline
$a_0$\rule{0mm}{4mm} & $1.053$ & $1.054$ & $1.055$ & $1.056$ \\ \hline
$x_0$\rule{0mm}{4mm} & $3,745,835,388$ & $2,605,443,747$ & $1,810,796,757$ & $1,220,594,340$ \\ \hline
\hline
$a_0$\rule{0mm}{4mm} & $1.057$ & $1.058$ & $1.059$ & $1.06$ \\ \hline
$x_0$\rule{0mm}{4mm} & $876,542,559$ & $673,828,570$ & $501,155,566$ & $383,446,375$ \\ \hline
\hline
$a_0$\rule{0mm}{4mm} & $1.061$ & $1.062$ & $1.063$ & $1.064$ \\ \hline
$x_0$\rule{0mm}{4mm} & $269,585,283$ & $196,894,353$ & $180,220,137$ & $116,749,925$ \\ \hline
\hline
$a_0$\rule{0mm}{4mm} & $1.065$ & $1.066$ & $1.067$ & $1.068$ \\ \hline
$x_0$\rule{0mm}{4mm} & $110,166,540$ & $76,223,058$ & $53,431,171$ & $46,097,944$ \\ \hline
\hline
$a_0$\rule{0mm}{4mm} & $1.069$ & $1.07$ & $1.071$ & $1.072$ \\ \hline
$x_0$\rule{0mm}{4mm} & $39,706,453$ & $31,027,247$ & $22,078,017$ & $18,339,738$ \\ \hline
\hline
$a_0$\rule{0mm}{4mm} & $1.073$ & $1.074$ & $1.075$ & $1.076$ \\ \hline
$x_0$\rule{0mm}{4mm} & $13,026,859$ & $12,895,928$ & $8,832,927$ & $7,299,254$ \\ \hline
\hline
$a_0$\rule{0mm}{4mm} & $1.077$ & $1.078$ & $1.079$ & $1.08$ \\ \hline
$x_0$\rule{0mm}{4mm} & $7,117,256$ & $5,465,656$ & $4,994,010$ & $3,462,478$ \\ \hline
\hline
$a_0$\rule{0mm}{4mm} & $1.081$ & $1.082$ & $1.083$ & $1.08366$ \\ \hline
$x_0$\rule{0mm}{4mm} & $3,455,648$ & $2,279,177$ & $1,529,630$ & $1,526,671$ \\ \hline
\hline
$a_0$\rule{0mm}{4mm} & $1.084$ & $1.085$ & $1.086$ & $1.087$ \\ \hline
$x_0$\rule{0mm}{4mm} & $1,525,432$ & $1,515,074$ & $1,200,014$ & $1,195,296$ \\ \hline
\hline
$a_0$\rule{0mm}{4mm} & $1.088$ & $1.089$ & $1.09$ & $1.091$ \\ \hline
$x_0$\rule{0mm}{4mm} & $624,878$ & $618,726$ & $618,058$ & $445,112$ \\ \hline
\hline
$a_0$\rule{0mm}{4mm} & $1.092$ & $1.093$ & $1.094$ & $1.095$ \\ \hline
$x_0$\rule{0mm}{4mm} & $359,804$ & $356203$ & $355,990$ & $355,177$ \\ \hline
\hline
$a_0$\rule{0mm}{4mm} & $1.096$ & $1.097$ & $1.098$ & $1.099$ \\ \hline
$x_0$\rule{0mm}{4mm} & $155,935$ & $155,907$ & $60,297$ & $60,224$ \\ \hline
\end{tabular} \ . \\
\end{center}
\end{kor}

\begin{proof}
Corollary \ref{kor404} implies that the inequality
\begin{equation}
\pi(x) < \frac{x}{\log x - 1.0344} \tag{8.1} \label{8.1}
\end{equation}
holds for every $x \geq 106,640,139,304,611$. If we compare the right-hand side of \eqref{8.1} with the integral logarithm $\text{li}(x)$, we can use Lemma \ref{lem402} to see that the required inequality \eqref{8.1} also holds for every $x$ with $98,269,667,551,459 \leq x \leq 106,640,139,304,611$. We conclude by direct computation.
\end{proof}

\begin{rema}
The real number $a_0 = 1.08366$ in Corollary \ref{kor801} is mostly only of historical value. On the basis of his study of a limited table of primes, Legendre stated 1808 (see \cite[p.\:394]{legendre2}) that
\begin{displaymath}
\pi(x) = \frac{x}{\log x - A(x)},
\end{displaymath}
where $\lim_{x \to \infty} A(x) = 1.08366$. Clearly Legendre's conjecture is equivalent to \eqref{1.8}. However, from \eqref{1.11}, it follows that the best value of $\lim_{x \to \infty} A(x)$ is 1. At this point it should be mentioned that Panaitopol \cite{pana1} claimed to have proved the inequality
\begin{equation}
\pi(x) < \frac{x}{\log x - 1.08366} \tag{8.2} \label{8.2}
\end{equation}
for every $x > 10^6$. In Corollary \ref{kor801}, it could be shown that $N = 1,526,671$ is the smallest possible positive integer so that the inequality \eqref{8.2} holds for every $x \geq N$. 
\end{rema}

Next, we consider the case where $m=1$. Here we obtain the following effective estimates for $\pi(x)$.

\begin{kor} \label{kor802}
We have
\begin{displaymath}
\pi(x) < \frac{x}{\log x - 1 - \frac{a_1}{\log x}}
\end{displaymath}
for every $x \geq x_1$, where
\begin{center}
\begin{tabular}{|l||c|c|c|c|c|c|c|c|c|c|}
\hline
$a_1$\rule{0mm}{4mm} & $1.11$ & $1.111$ & $1.112$ & $1.113$\\ \hline
$x_1$\rule{0mm}{4mm} & $62,998,850,942,976$ & $49,246,036,992,716$ & $38,472,138,880,411$ & $30,658,643,813,468$\\ \hline
\hline
$a_1$\rule{0mm}{4mm} & $1.114$ & $1.115$ & $1.116$ & $1.117$\\ \hline
$x_1$\rule{0mm}{4mm} & $23,767,640,743,883$ & $19,278,513,358,342$ & $15,142,627,022,527$ & $12,279,648,138,508$ \\ \hline
\hline
$a_1$\rule{0mm}{4mm} & $1.118$ & $1.119$ & $1.12$ & $1.121$\\ \hline
$x_1$\rule{0mm}{4mm} & $9,684,114,630,824$ & $7,981,446,192,206$ & $6,323,967,140,812$ & $5,273,225,700,761$ \\ \hline
\hline
$a_1$\rule{0mm}{4mm} & $1.122$ & $1.123$ & $1.124$ & $1.125$\\ \hline
$x_1$\rule{0mm}{4mm} & $4,170,462,893,841$ & $3,458,549,136,539$ & $2,825,539,807,244$ & $2,292,448,124,593$ \\ \hline
\hline
$a_1$\rule{0mm}{4mm} & $1.126$ & $1.127$ & $1.128$ & $1.129$\\ \hline
$x_1$\rule{0mm}{4mm} & $1,903,596,231,542$ & $1,573,767,234,188$ & $1,290,096,268,844$ & $1,073,403,839,693$ \\ \hline
\hline
$a_1$\rule{0mm}{4mm} & $1.13$ & $1.131$ & $1.132$ & $1.133$\\ \hline
$x_1$\rule{0mm}{4mm} & $889,377,392,161$ & $782,989,678,664$ & $608,408,258,090$ & $540,050,850,157$ \\ \hline
\hline
$a_1$\rule{0mm}{4mm} & $1.134$ & $1.135$ & $1.136$ & $1.137$\\ \hline
$x_1$\rule{0mm}{4mm} & $452,875,824,702$ & $373,479,021,700$ & $335,562,521,091$ & $263,728,502,964$ \\ \hline
\hline
$a_1$\rule{0mm}{4mm} & $1.138$ & $1.139$ & $1.14$ & $1.141$\\ \hline
$x_1$\rule{0mm}{4mm} & $242,118,904,367$ & $201,924,836,111$ & $161,054,192,492$ & $149,061,190,565$ \\ \hline
\hline
$a_1$\rule{0mm}{4mm} & $1.142$ & $1.143$ & $1.144$ & $1.145$\\ \hline
$x_1$\rule{0mm}{4mm} & $125,233,112,846$ & $105,053,836,224$ & $86,061,321,374$ & $77,278,924,451$ \\ \hline
\hline
$a_1$\rule{0mm}{4mm} & $1.146$ & $1.147$ & $1.148$ & $1.149$\\ \hline
$x_1$\rule{0mm}{4mm} & $61,344,524,412$ & $57,720,831,343$ & $46,039,922,948$ & $42,575,222,481$ \\ \hline
\end{tabular} \ . \\
\end{center}
\end{kor}

\begin{proof}
The proof is quite similar to the proof of Corollary \ref{kor801} and we leave the details to the reader.
\end{proof}

Finally, we consider the case where $m=2$ and find the following explicit estimates for $\pi(x)$ .

\begin{kor} \label{kor803}
We have
\begin{displaymath}
\pi(x) < \frac{x}{\log x - 1 - \frac{1}{\log x} - \frac{a_2}{\log^2x}}
\end{displaymath}
for every $x \geq x_2$, where
\begin{center}
\begin{tabular}{|l||c|c|c|c|c|c|c|c|c|c|}
\hline
$a_2$\rule{0mm}{4mm} & $3.49$ & $3.5$ & $3.51$ & $3.52$ \\ \hline
$x_2$\rule{0mm}{4mm} & $83,027,761,686,134$ & $50,794,512,296,846$ & $30,594,003,254,258$ & $17,348,455,129,950$ \\ \hline
\hline
$a_2$\rule{0mm}{4mm} & $3.53$ & $3.54$ & $3.55$ & $3.56$ \\ \hline
$x_2$\rule{0mm}{4mm} & $11,655,963,556,138$ & $5,539,984,798,515$ & $4,489,052,430,063$ & $2,180,930,569,481$ \\ \hline
\hline
$a_2$\rule{0mm}{4mm} & $3.57$ & $3.58$ & $3.59$ & $3.6$ \\ \hline
$x_2$\rule{0mm}{4mm} & $1,464,200,206,021$ & $882,055,689,961$ & $584,256,118,105$ & $437,882,804,654$ \\ \hline
\hline
$a_2$\rule{0mm}{4mm} & $3.61$ & $3.62$ & $3.63$ & $3.64$ \\ \hline
$x_2$\rule{0mm}{4mm} & $332,203,763,508$ & $201,890,631,296$ & $148,632,348,138$ & $102,965,110,268$ \\ \hline
\hline
$a_2$\rule{0mm}{4mm} & $3.65$ & $3.66$ & $3.67$ & $3.68$ \\ \hline
$x_2$\rule{0mm}{4mm} & $55,102,251,180$ & $38,278,086,931$ & $24,178,954,639$ & $21,729,109,565$ \\ \hline
\end{tabular} \ .
\end{center}
\end{kor}

\begin{proof}
Similar to Corollary \ref{kor801}.
\end{proof}

\section*{Acknowledgement}
I would like to express my great appreciation to Kim Walisch and Thomas Lessmann for the support in writing the C++ codes used in this paper. Furthermore I thank Samuel Broadbent, Habiba Kadiri, Allysa Lumley, Nathan Ng, and Kirsten Wilk, whose paper has motivated me to deal with the present topic again. Moreover, I would also like to thank the two beautiful souls R. and O. for the never ending inspiration.

\end{document}